\documentclass[draft]{amsart}

\newtheorem{thm}{Theorem}[section]

\newtheorem{prop}[thm]{Corollary}
\newtheorem{lem}[thm]{Lemma}

\theoremstyle{remark}

\theoremstyle{definition}

\numberwithin{equation}{section}
\numberwithin{thm}{section}

\begin{document}


\title{On the K\"ahler structures over Quot schemes}

\author{Indranil Biswas}
\address{School of Mathematics, Tata Institute of Fundamental
Research, Homi Bhabha Road, Bombay 400005, India
}

\email{indranil@math.tifr.res.in}

\author{Harish Seshadri}

\address{Indian Institute of Science, Department of Mathematics,
Bangalore 560003, India
}

\email{harish@math.iisc.ernet.in}

\thanks{The first-named author is supported by
the J. C. Bose Fellowship.}

\begin{abstract}
Let $S^n(X)$ be the $n$-fold symmetric product of a compact connected Riemann surface
$X$ of genus $g$ and gonality $d$. We prove that $S^n(X)$ admits a K\"ahler structure
such that all the holomorphic bisectional curvatures are nonpositive if and only if $n\,
<\, d$. Let ${\mathcal Q}_X(r,n)$ be the Quot scheme parametrizing the torsion
quotients of ${\mathcal O}^{\oplus r}_X$ of degree $n$. If $g\, \geq\, 2$ and $n\,
\leq\, 2g-2$, we prove that ${\mathcal Q}_X(r,n)$ does not admit a K\"ahler structure
such that all the holomorphic bisectional curvatures are nonnegative.
\end{abstract}

\maketitle

\section{Introduction}

Let $X$ be a compact connected Riemann surface of genus $g$ and gonality $d$. For a
positive integer $n$, let $S^n(X)$ denote the $n$-fold symmetric product of $X$. More
generally, ${\mathcal Q}_X(r,n)$ will denote that Quot scheme that parametrizes all
the torsion quotients of ${\mathcal O}^{\oplus r}_X$ of degree $n$.
So $S^n(X)\,=\, {\mathcal Q}_X(1,n)$. This ${\mathcal Q}_X(r,n)$ is a complex
projective manifold.

We prove the following (see Theorem \ref{thm1}):

\textit{The symmetric product $S^n(X)$ admits a
K\"ahler structure satisfying the condition that all the holomorphic
bisectional curvatures are nonpositive if and only if $n\, <\, d$.}

The ``only if'' part was proved in \cite{Bi1}.

The main theorem of \cite{BR} says the following (see \cite[Theorem 1.1]{BR}):
If $g\, \geq\, 2$ and $n\, \leq\, 2(g-1)$, then $S^n(X)$ does not admit any K\"ahler
metric for which all the holomorphic bisectional curvatures are nonnegative.
A simpler proof of this result was given in \cite{Bi2}. Here we prove the
following generalization of it (see Proposition \ref{thm2}):

\textit{Assume that $g\, \geq\, 2$ and $n\, \leq\, 2(g-1)$. Then ${\mathcal Q}_X
(r,n)$ does not admit any K\"ahler structure such that all the holomorphic
bisectional curvatures are nonnegative.}

If $r\, >\, 1$, the method in \cite{Bi2} give a much weaker version of
Proposition \ref{thm2}.

\medskip
\noindent
\textbf{Acknowledgements.} We thank the referee for helpful comments.

\section{Preliminaries}

Let $X$ be a compact connected Riemann surface of genus $g$.
For any positive integer $n$, consider the Cartesian product $X^n$.
Denote by $P_n$ the group of permutations of $\{1\, , \cdots \, ,n\}$.
The group $P_n$ has a natural action on $X^n$. The quotient $X^n/P_n$
will be denoted by $S^n(X)$; it is called the $n$-\textit{fold
symmetric product} of $X$.

Let ${\mathcal O}_X$ denote the sheaf of germs of holomorphic functions on $X$.
For a positive integer $r$, consider the sheaf ${\mathcal O}^{\oplus r}_X$ of
germs of holomorphic sections of the trivial holomorphic vector bundle on $X$ of
rank $r$. For any positive integer $n$, let
$$
{\mathcal Q}\, :=\, {\mathcal Q}_X(r,n)
$$
be the Quot scheme parametrizing all the torsion quotients of degree $n$
of the ${\mathcal O}_X$--module ${\mathcal O}^{\oplus
r}_X$. Equivalently, points of ${\mathcal Q}$ parametrize coherent analytic
subsheaves of ${\mathcal O}^{\oplus r}_X$ of rank $r$ and degree $-n$. This
${\mathcal Q}$ is an irreducible smooth complex projective variety of
dimension $rn$ \cite[p. 1, Theorem 2]{Be}.

Note that ${\mathcal Q}_X(1,n)$ is identified with the symmetric
product $S^n(X)$ by sending a quotient of ${\mathcal O}_X$ to the
scheme--theoretic support of it. If we consider ${\mathcal Q}_X(1,n)$ as
the parameter space for the coherent analytic subsheaves of ${\mathcal O}_X$
of rank $1$ and degree $-n$, then the above identification of
${\mathcal Q}_X(1,n)$ with $S^n(X)$ sends a subsheaf
$\psi\, :\, L\, \hookrightarrow\, {\mathcal O}_X$ to the divisor of $\psi$.

The \textit{gonality} of $X$ is the
smallest integer $d$ such that there is a nonconstant holomorphic map
$X\, \longrightarrow\, {\mathbb C}{\mathbb P}^1$ of degree $d$
(see \cite[p. 171]{Ei}). Therefore, the gonality
of $X$ is one if and only if $g\,=\, 0$. If $g\, \in\, \{1\, ,2\}$, then
the gonality of $X$ is two. More generally, the gonality of $X$ is two if and
only if $X$ is hyperelliptic of positive genus.

\section{Nonpositive holomorphic bisectional curvatures}

\begin{thm}\label{thm1}
Let $d$ denote the gonality of $X$. The symmetric product $S^n(X)$
admits a K\"ahler structure satisfying the condition that all the holomorphic
bisectional curvatures are nonpositive if and only if $n\, <\, d$.
\end{thm}

\begin{proof}
If $n\, \geq \, d$, then we know that $S^n(X)$ does not
admit any K\"ahler structure such that all the holomorphic bisectional
curvatures are nonpositive \cite[p. 1491, Proposition 3.2]{Bi1}. We recall
that this follows from the fact that there is a nonconstant holomorphic
embedding of ${\mathbb C}{\mathbb P}^1$ in $S^n(X)$ if $n\, \geq \, d$.

So assume that $n\, <\, d$.

Let
\begin{equation}\label{vp}
\varphi\, :\, S^n(X)\,\longrightarrow\, \text{Pic}^n(X)
\end{equation}
be the natural holomorphic map that sends any $\{x_1\, ,\cdots\, ,x_n\}
\,\in\, S^n(X)$ to the holomorphic line bundle ${\mathcal O}_X(\sum_{i=1}^n x_i)$.
We will show that $\varphi$ is an immersion.

Take any point $\underline{x}\, =\, \{x_1\, ,\cdots\, ,x_n\}\,\in\, S^n(X)$.
The divisor $\sum_{i=1}^n x_i$ will be denoted by $D$. Let
\begin{equation}\label{D}
0\, \longrightarrow\, {\mathcal O}_X(-D)\, \longrightarrow\, {\mathcal O}_X
\, \longrightarrow\, Q'(\underline{x})\,:=\, {\mathcal O}_X/{\mathcal O}_X(-D)
\, \longrightarrow\, 0
\end{equation}
be the short exact sequence corresponding to the point $\underline{x}$.
tensoring it with ${\mathcal O}_X(-D)^*\,=\, {\mathcal O}_X(D)$ we get the short exact
sequence
$$
0\, \longrightarrow\, End({\mathcal O}_X(-D))\,=\,{\mathcal O}_X
\, \longrightarrow\, Hom({\mathcal O}_X(-D)\, ,{\mathcal O}_X)\,=\,
{\mathcal O}_X(D)
$$
$$
\longrightarrow\, Q(\underline{x})\,:=\,
Hom({\mathcal O}_X(-D)\, ,Q'(\underline{x}))\, \longrightarrow\, 0\, .
$$
Let
\begin{equation}\label{e1}
0\, \longrightarrow\, H^0(X,\, {\mathcal O}_X)\,\stackrel{\alpha}{\longrightarrow}\,
H^0(X,\, {\mathcal O}_X(D))\,\stackrel{\beta}{\longrightarrow}\, H^0(X,\,
Q(\underline{x}))\,\stackrel{\gamma}{\longrightarrow}\, H^1(X,\, {\mathcal O}_X)
\end{equation}
be the long exact sequence of cohomologies associated to this short exact sequence
of sheaves.

The holomorphic tangent space to $S^n(X)$ at $\underline{x}$ is
$$
T_{\underline{x}}S^n(X)\,=\, H^0(X,\, Q(\underline{x}))\, ,
$$
and the tangent bundle of $\text{Pic}^n(X)$ is the trivial vector bundle
with fiber $H^1(X,\, {\mathcal O}_X)$. The differential at $\underline{x}$
of the map $\varphi$ in \eqref{vp}
$$
(d\varphi)(\underline{x})\, :\, T_{\underline{x}}S^n(X)\,=\, H^0(X,\,
Q(\underline{x}))\,\longrightarrow\, T_{\varphi(\underline{x})}\text{Pic}^n(X)\,=\,
H^1(X,\, {\mathcal O}_X)
$$
satisfies the identity
\begin{equation}\label{e2}
(d\varphi)(\underline{x})\,=\, \gamma\, ,
\end{equation}
where $\gamma$ is the homomorphism in \eqref{e1}.

Now, $H^0(X,\, {\mathcal O}_X)\,=\, \mathbb C$. Since $n\, <\, d$, it can be
shown that
$$
H^0(X,\, {\mathcal O}_X(D))\,=\, \mathbb C\, .
$$
Indeed, $\dim H^0(X,\, {\mathcal O}_X(D))\, \geq\, 1$ because $D$ is effective.
If
$$
\dim H^0(X,\, {\mathcal O}_X(D))\, \geq\, 2\, ,
$$
then considering the partial linear system given by two linearly
independent sections of ${\mathcal O}_X(D)$ we get a holomorphic map from
$X$ to ${\mathbb C}{\mathbb P}^1$ whose degree coincides with the degree of $D$.
This contradicts the fact that the gonality of $X$ is strictly bigger than $n$.
Therefore, $H^0(X,\, {\mathcal O}_X(D))\,=\, \mathbb C$.

Since $H^0(X,\, {\mathcal O}_X(D))\,=\, \mathbb C$,
the homomorphism $\alpha$ in \eqref{e1} is an isomorphism. Hence $\beta$ in
the exact sequence \eqref{e1} is the zero homomorphism and $\gamma$ in \eqref{e1}
is injective.

Since $\gamma$ in \eqref{e1} is injective, from \eqref{e2} we conclude that
$\varphi$ is an immersion.

The compact complex torus $\text{Pic}^n(X)$ admits a flat K\"ahler metric
$\omega$. The pullback $\varphi^*\omega$ is a K\"ahler metric on $S^n(X)$
because $\varphi$ is an immersion. Since $\omega$ is flat, all the 
holomorphic bisectional curvatures of $\varphi^*\omega$ are nonpositive.
\end{proof}

\begin{lem}\label{lem1}
Take $r\, \geq\, 2$ and take any positive integer $n$. 
The Quot scheme ${\mathcal Q}_X(r,n)$ does not
admit any K\"ahler structure such that all the holomorphic bisectional
curvatures are nonpositive.
\end{lem}

\begin{proof}
Let ${\mathcal K} \longrightarrow\, {\mathcal O}^{\oplus r}_{X\times
{\mathcal Q}_X(r,n)}$ be the universal homomorphism on $X\times
{\mathcal Q}_X(r,n)$. Let
$$
\bigwedge\nolimits^r{\mathcal K} \longrightarrow\,
\bigwedge\nolimits^r{\mathcal O}^{\oplus r}_{X\times
{\mathcal Q}_X(r,n)}\,=\, {\mathcal O}_{X\times
{\mathcal Q}_X(r,n)}
$$
be the top exterior product. This homomorphism produces a holomorphic map
\begin{equation}\label{f}
f\,:\, {\mathcal Q}_X(r,n)\,\longrightarrow\, S^n(X)\, .
\end{equation}
Take any $\underline{x}\, =\, \{x_1\, ,\cdots\, ,x_n\}\,\in\,
S^n(X)$ such that all $x_i$ are distinct points. Then $f^{-1}(\underline{x})$ is
isomorphic to $({\mathbb C}{\mathbb P}^{r-1})^n$. In particular, there are
embeddings of ${\mathbb C}{\mathbb P}^{1}$ in ${\mathcal Q}_X(r,n)$. This
immediately implies that ${\mathcal Q}_X(r,n)$ does not
admit any K\"ahler structure such that all the holomorphic bisectional
curvatures are nonpositive.
\end{proof}

\section{Nonnegative holomorphic bisectional curvatures}

In this section we assume that $g\, \geq\, 2$.

\begin{prop}\label{thm2}
If $n\, \leq\, 2(g-1)$, then ${\mathcal Q}_X(r,n)$ does not admit any K\"ahler
structure such that all the holomorphic bisectional curvatures are nonnegative.
\end{prop}

\begin{proof}
Assume that ${\mathcal Q}_X(r,n)$ has a K\"ahler structure $\omega$
such that all the holomorphic bisectional curvatures for $\omega$ are nonnegative.
Consequently, tangent bundle $T{\mathcal Q}_X(r,n)$ is nef.
See \cite[p. 305, Definition 1.9]{DPS} for the definition of a nef vector
bundle; nef line bundles are introduced in \cite[p. 299, Definition 1.2]{DPS}.
Since ${\mathcal Q}_X(r,n)$ is a complex projective manifold, nef bundles
on ${\mathcal Q}_X(r,n)$ in the sense of \cite{DPS} coincide with the
nef bundles on ${\mathcal Q}_X(r,n)$ in the algebraic geometric sense (see lines 13--14
(from top) in \cite[p. 296]{DPS}).

Let
\begin{equation}\label{de}
\delta\,:=\, \varphi\circ f\,:\, {\mathcal Q}_X(r,n) \,\longrightarrow
\,\text{Pic}^n(X)
\end{equation}
be the composition, where $\varphi$ and $f$ are constructed in \eqref{vp}
and \eqref{f} respectively. The homomorphism
$$
H^1(\text{Pic}^n(X),\, {\mathbb Q})\,\longrightarrow\, H^1(S^n(X),\, {\mathbb Q})\, ,
~\ c\, \longmapsto\, \varphi^* c
$$
is an isomorphism \cite[p. 325, (6.3)]{Ma}. Also, the homomorphism
$$
H^(S^n(X),\, {\mathbb Q})\,\longrightarrow\, H^1({\mathcal Q}_X(r,n),\,
{\mathbb Q})\, ,~\ c\, \longmapsto\, f^* c
$$
is an isomorphism \cite[p. 647, Proposition 4.2]{BGL} (see also the
last line of \cite[p. 647]{BGL}). Combining these we conclude that the homomorphism
$$
H^1(\text{Pic}^n(X),\, {\mathbb Q})\,\longrightarrow\,H^1({\mathcal Q}_X(r,n),\,
{\mathbb Q})\, ,~\ c\, \longmapsto\, \delta^* c
$$
is an isomorphism, where $\delta$ is constructed in \eqref{de}.

Since the fibers of $\delta$ are connected, this implies
that $\delta$ is the Albanese morphism for ${\mathcal Q}_X(r,n)$.
Since the tangent bundle of ${\mathcal Q}_X(r,n)$ is nef, the
Albanese map $\delta$ is a holomorphic surjective submersion onto
$\text{Pic}^n(X)$ \cite{CP}, \cite[p. 321, Proposition 3.9]{DPS}.

Since $\delta$ is surjective, the map $\varphi$ in \eqref{vp} is surjective.
Therefore,
$$
g\,=\, \dim \text{Pic}^n(X)\, \leq\, \dim S^n(X)\,=\, n\, .
$$
The map $\varphi$ is a submersion because $\delta$ is a submersion and
$f$ is surjective.

We will show that $\varphi$ is not a submersion if $n\, \leq\, 2(g-1)$.

Take any $n\, \leq\, 2(g-1)$. Let $D'$ be the divisor of a holomorphic
$1$--form on $X$. We note that the degree of $D'$ is $2(g-1)$. Take an
effective divisor $D$ on $X$ of degree $n$ such that $D'-D$ is effective. Writing
$D'\,=\, x_1+\ldots +x_{2g-2}$, we may take $D\,=\, x_1+\ldots +x_n$. Substitute
this $D$ in \eqref{D} and consider the corresponding long exact sequence of
cohomologies
\begin{equation}\label{f1}
H^0(X,\,
Q(\underline{x}))\,\stackrel{\gamma}{\longrightarrow}\, H^1(X,\, {\mathcal O}_X)
\,\stackrel{\gamma'}{\longrightarrow}\, H^1(X,\, {\mathcal O}_X(D))
\,\longrightarrow\, H^1(X,\, Q(\underline{x}))
\end{equation}
as in \eqref{e1}. We note that $H^1(X,\, Q(\underline{x}))\,=\,0$ because
$Q(\underline{x})$ is a torsion sheaf on $X$. From Serre duality,
$$
H^1(X,\, {\mathcal O}_X(D))\,=\, H^0(X,\, {\mathcal O}_X(D'-D))^*\, .
$$
Now $H^0(X,\, {\mathcal O}_X(D'-D))\,\not=\, 0$ because $D'-D$ is effective.
Combining these, we conclude that $\gamma'$ in \eqref{f1} is nonzero. Hence
$\gamma$ in \eqref{f1} is not surjective. Therefore, from \eqref{e2} we conclude that
the differential $d\varphi$ of $\varphi$ is not surjective at the point $D\,\in\,S^n(X)$.
In particular, $\varphi$ is not a submersion if $n\, \leq\, 2(g-1)$. This completes
the proof.
\end{proof}

\end{document}